\documentclass[12pt]{amsart}

\usepackage[english]{babel}

\usepackage[footskip=.2in]{geometry}
\usepackage{amsmath,amssymb,amscd,amsfonts,mathtools,tikz,caption,subcaption}
\usetikzlibrary{patterns}
\usepackage[linkcolor=blue,colorlinks=true,citecolor=blue]{hyperref}
\usepackage[capitalise]{cleveref}
\usepackage{mathrsfs}

\usepackage{pgfplots}
\pgfplotsset{compat=1.14}
\usepackage{tikz}
\usepackage{tikz-cd}

\usetikzlibrary{quotes,angles,arrows,calc}

\usepackage{hyperref}

\newtheorem{thm}{Theorem}[section]
\newtheorem{lem}[thm]{Lemma}
\newtheorem{defn}[thm]{Definition}
\newtheorem{prop}[thm]{Proposition}
\newtheorem{coro}[thm]{Corollary}
\newtheorem{rmk}[thm]{Remark}

\newtheorem{example}[thm]{Example}

\numberwithin{equation}{section}

\newcommand{\R}{\mathbb{R}}
\newcommand{\C}{\mathbb{C}}
\newcommand{\T}{\mathbb{T}}

\newcommand{\Q}{\mathbb{Q}}

\newcommand{\N}{\mathbb{N}}

\newcommand{\cC}{\mathcal{C}}
\newcommand{\cA}{\mathcal{A}}

\newcommand{\cD}{\mathcal{D}}
\newcommand{\cE}{\mathcal{E}}

\newcommand{\cL}{\mathcal{L}}

\newcommand{\cM}{\mathcal{M}}

\newcommand{\defeq}{\vcentcolon=}

\newcommand{\rn}{\mathbb{R}^n}

\newcommand{\bbm}{\begin{bmatrix}}
\newcommand{\ebm}{\end{bmatrix}}
\newcommand{\bpm}{\begin{pmatrix}}
\newcommand{\epm}{\end{pmatrix}}
\newcommand{\bsm}{\left(\begin{smallmatrix}}
\newcommand{\esm}{\end{smallmatrix}\right)}
\newcommand{\bsbm}{\left[\begin{smallmatrix}}
\newcommand{\esbm}{\end{smallmatrix}\right]}

\newcommand{\epi}{\mathrm{epi}}

\begin{document}
\title{A study in quantitative equidistribution on the unit square}
\author{Max Goering}
\address{Max-Planck-Institut f\"ur Mathematik in den Naturwissenschaften, Inselstr. 22, 04103 Leipzig, Germany}
\email{goering@mis.mpg.de}
\author{Christian Weiss}
\address{Ruhr West University of Applied Sciences, Duisburger Str. 100, 45479 M\"ulheim an der Ruhr, Germany}
\email{christian.weiss@hs-ruhrwest.de}

\begin{abstract}
	The distributional properties of the translation flow on the unit square have been considered in different fields of mathematics, including algebraic geometry and discrepancy theory. One method to quantify equidistribution is to compare the error between the actual time the translation flow spent in specific sets $E \subset [0,1]^2$ to the expected time. In this article, we prove that when $E$ is in the algebra generated by convex sets the error is of order at most $\log(T)^{1+\varepsilon}$ for almost every direction. For all but countably many badly approximable directions, the bound can be sharpened to $\log(T)^{1/2+\varepsilon}$. The error estimates we produce are smaller than for general measurable sets as proved by Beck, while our class of examples is larger than in the work of Grepstad-Larcher who obtained the bounded remainder property for their sets. Our proof relies on the duality between local convexity of the boundary and regularity of sections of the flow.
\end{abstract}
\maketitle

\section{Introduction} In this paper, we are interested in the two-dimensional translation flow, $Y^{x}_{\alpha}$, in direction $\alpha$ on the two-dimensional torus $\T^2$, i.e. the unit square $[0,1]^2$ with opposite sides glued. Given a starting point $x=(x_1,x_2) \in [0,1]^2$, $Y^x_{\alpha}: [0,\infty) \to [0,1]^2$ is defined by 
$$Y^x_{\alpha}(t) = \left( \{x_1+t\} , \{x_2+\alpha t\} \right),$$
where $\{ \cdot \}$ denotes the fractional part of a real number. Our focus here is on distributional properties of the flow. The two-dimensional Kronecker-Weyl equidistribution theorem states that the orbit of a translation flow on the unit square is uniformly distributed if $\alpha$ is irrational and otherwise it is periodic. The translation flow on the square may also be regarded as a flow on the corresponding Riemann surface which is obtained by gluing opposite edges of the square. 

In order to describe the distributional properties of the translation flow quantitatively, it is a natural approach to compare the actual time spent in specific sets $E \subset [0,1]^2$ to the expected time. According to Beck, \cite[Theorem 1]{Bec15}, the following superuniformity result holds for arbitrary measurable sets.
\begin{thm}[Beck, 2015] \label{thm:beck} Let $E \subset [0,1]^2$ be an arbitrary Lesbesgue measurable set with $\lambda_{2}(E)>0$. Then for every $\varepsilon > 0$, almost all $\alpha > 0$ and every starting point $(x_1,x_2) \in [0,1]^2$, we have
	$$\int_{0}^T \chi_E(Y_{\alpha}^{x}(t)) \mathrm{d} t -  T\lambda_{2}(E) = o(\log(T)^{3 + \varepsilon}).$$
\end{thm} 
Note that Beck's result does not impose any further geometric restrictions on $E$ besides being measurable. In face of the generality, the error term is surprisingly, or in Beck's words \textit{shockingly} small. In fact, Beck even proved a more general version of the theorem, extending it to the case of arbitrary real-valued square-integrable functions. The proof covers more than 30 pages in \cite{Bec15} and uses techniques from Fourier analysis, the proof of Weyl's criterion, lattice-point counting, and averaging arguments.

 If geometric restrictions are enforced, stronger results with even smaller error terms can be derived. To describe them we introduce the following notion, compare \cite{Fer92,GL16}.
\begin{defn} \label{def:bounded_remainder_set}  Let $E \subset [0,1]^2$ be an arbitrary measurable subset of the unit square. We say that $E$ is a \textbf{bounded remainder set} for the two-dimensional translation flow with starting slope $\alpha$ and starting point $x \in [0,1]^2$ if the error satisfies
	$$\int_{0}^T \chi_E(Y_{\alpha}^{x}(t)) \mathrm{d} t - T \lambda_{2}(E) = O(1),$$
i.e. it is uniformly bounded for all $T > 0$.

\end{defn}
It is worth noting that the constant in Definition~\ref{def:bounded_remainder_set} may depend on $E$ as well as on $\alpha$ and $x$. The following bounded remainder sets were discovered in \cite{GL16}. 
\begin{thm}[\cite{GL16}, Theorem~1.3 and Theorem~1.4] \label{thm:GL} \ \\ \
 \vspace{-.25in}
 \begin{itemize}	
		\item For almost all $\alpha > 0$	and every $x \in [0,1]^2$, every polygon $E \subset [0,1]^2$ with no edge of slope $\alpha$ is a bounded remainder set for the translation flow $Y_{\alpha}^x(t)$.
		\item For almost all $\alpha > 0$ and every $x \in [0,1]^2$, every convex set $E \subset [0,1]^2$ whose boundary $\partial E$ is a twice continuously differentiable curve with positive curvature at every point is a bounded remainder set for the translation flow $Y_{\alpha}^x(t)$.
	\end{itemize}
\end{thm}
The proof of Theorem~\ref{thm:GL} is based on reducing the calculation of the error term to a one-dimensional problem and a clever application of the classical Koksma-Hlawka inequality involving a rational approximation of Kronecker sequences by means of the Ostrowski expansion. In this paper, we propose another approach to bound the error term of the time discrepancy. 
Unlike the methods in \cite{GL16}, our technique depends upon the local geometry of the set. This {gives more flexibility and allows} us to consider any set $E \subset [0,1]^{2}$ such that $E$ can be written as the finite union of convex sets and their complements, that is, whenever $E$ is in the \textbf{algebra of convex sets}, denoted $\cE([0,1]^2)$. Similarly we let $\cE_{\sigma}([0,1]^{2})$ denote the algebra of sets generated by $\sigma$-convex sets, see Definition \ref{def:sigma-convex}. Under this assumption, we discover new results analogous to \cite{Bec15} and \cite{GL16} with an error term of intermediate order.

\begin{thm} \label{thm:error_term}
If $E \in \cE ([0,1]^{2})$ then for almost every slope $\alpha \in \R$:
\begin{enumerate}
\item[(A)] If $\varepsilon > 0$ then,
\begin{equation} \label{e:rb}
\int_{0}^{T} \chi_{E}(Y^{x}_{\alpha}(t)) dt - T \lambda_{2}(E) = o \left( \log(T)^{1 + \varepsilon}\right).
\end{equation}

\item[(B)] If $\alpha \not \in \cD_{2}(E)$\footnote{See definition \ref{d:degenerate}} is badly-approximable then \eqref{e:rb} can be sharpened to
\begin{equation} \label{e:barb}
\int_{0}^{T} \chi_{E}(Y^{x}_{\alpha}(t))dt - T \lambda_{2}(E) = o\left(\log(T)^{\frac{1}{2} + \varepsilon}\right) \qquad \forall \varepsilon > 0.
\end{equation}

\end{enumerate}
If $E \in \cE_{\sigma}([0,1]^{2})$ for some $\sigma \in [2,\infty)$, 
\begin{enumerate}
\item[(C)] For almost every $\alpha \in \R$ and all $\varepsilon > 0$, \eqref{e:rb} holds.
\item[(D)] If $\alpha$ is badly-approximable, then for all $\varepsilon > 0$
\begin{equation*}
\int_{0}^{T} \chi_{E}(Y^{x}_{\alpha}(t))dt - T \lambda_{2}(E) = 
\begin{cases}
o \left( \log(T)^{\frac{1}{2} + \varepsilon} \right) & \alpha \not \in \cD_{2}(E) \\
o \left( \log(T)^{\frac{1}{\sigma^{\prime}} + \varepsilon} \right) & \alpha \in \cD_{2}(E),
\end{cases}
\end{equation*}
where $\sigma^{\prime}$ denotes the H\"older conjugate of $\sigma$.
\end{enumerate}
\end{thm}

We note that all of the errors in Theorem \ref{thm:error_term} are smaller than in Theorem~\ref{thm:beck} (from \cite{Bec15}) but larger than in Theorem \ref{thm:GL} (from \cite{GL16}). This is reasonable because the classes $\cE([0,1]^{2})$ and $\cE_{\sigma}([0,1]^{2})$ are more general than those considered in \cite{GL16}, but less general than those in \cite{Bec15}. In particular, Theorem \ref{thm:beck} holds for Lebesgue measurable sets, which can be recognized as the completion (with respect to the Lebesgue measure) of the $\sigma$-algebra generated by convex sets, instead of the \textit{finite} algebra generated by convex sets, $\cE([0,1]^{2})$. In early preparation of these works, we expected to strengthen Theorem~\ref{thm:error_term} to hold for all irrational directions for which the convexity of $\partial E$ was not too degenerate. Hence, we expected all but countably many directions to be covered. However, this is not possible to achieve, the reasons for which are outlined in Remark \ref{rmk:countable}.

On the other hand, in \cite[Theorem 1.4.1]{BDY1} it is shown that for convex sets \eqref{e:rb} holds for some constant \textit{independent of the convex set}. Compared to the techniques herein, bounds independent of the convex set are striking. But, one can easily show that the independence on the set in \cite[Theorem 1.4.1]{BDY1} cannot be extended to the class of sets $\cE([0,1]^{2})$\footnote{For instance, by taking $E$ to be the finite union of line segments with slope $\alpha$.}. Hence we ask,

\underline{Open Question 1:} If $E \in \cE([0,1]^{2})$ does \eqref{e:rb} hold with a constant depending on $\#| \cC|$\footnote{See Definition \ref{d:bp} and Lemma \ref{l:decomp}} and $\alpha$, but not on $E$ itself?

Finally, we note that in \cite[Corollary 2 of Theorem 1.4.1]{BDY1} it is shown that when $\alpha$ is a quadratic irrational slope the error term is like $O\left( \log (T) \right)$, again, independent of the convex set. Since the class of badly-approximable $\alpha$ contains the set of quadratic irrationals (see \cite[Chapter 3.3]{einsiedler2013ergodic} and references therein) it follows that, after throwing out all {degenerate slopes} $\alpha \in \cD_{2}(E)$-- merely a countable set of directions-- Theorem \ref{thm:error_term}(B) is a strict improvement on both the asymptotics and the class of sets. At first glance, parts (B) and (D) of Theorem~\ref{thm:error_term} might seem to contradict the lower bound given in \cite[Theorem 1.4.1]{BDY1}, where it is claimed that for any direction $\alpha$ there exist a convex set $C_0$ such that the error term is of order  at least $\log(T)$. However, the choice of $C_0$ (a parallelogram for which two sides are parallel to $\alpha$) depends on $\alpha$ and in particular has $\alpha \in \cD_{2}(C_{0})$.

The strict asymptotic improvement in Theorem \ref{thm:error_term}(B) comes from applying an $L^{p}$-Koksma-Hlawka inequality instead of the classical Koksma-Hlawka inequality. For this reason, we need to be able to verify the technical assumption that $\tau_{E,\alpha}$ is in the Sobolev space $W^{1,s}(\T)$, see Section \ref{subsec:background} for details. This result is stated formally in Theorem \ref{t:sets} due to its independent interest.

\begin{thm} \label{t:sets}
Fix $E \in \cE([0,1]^{2})$. For any $\sigma \in [2,\infty)$ and any $\alpha \not \in \cD_{\sigma}(E)$\footnote{See Definition \ref{d:degenerate}.} it holds $\tau_{E,\alpha} \in W^{1,s}(\T)$ for all $s < \frac{1}{\sigma^{\prime}}$. In particular, if there exists $\sigma\in [2,\infty)$ so that $E \in \cE_{\sigma}([0,1]^{2})$ then for all $\alpha \in \R$ and all $s < \frac{1}{\sigma^{\prime}}$ it holds $\tau_{E,\alpha} \in W^{1,s}(\T)$. 
\end{thm}

If $\sigma < \tilde{\sigma}$ then $\cD_{\tilde{\sigma}} \subset \cD_{\sigma}$ and Corollary \ref{c:sard} ensures  that for all $\sigma \ge 2$, the set $\cD_{\sigma}(E)$ is a countable set for all $E \in \cE([0,1]^{2})$.

In light of Lemma \ref{l:decomp}, one can read Theorem \ref{t:sets} as a statement about quantitative duality between local convexity (or concavity) of $\partial E$ and regularity on $\tau_{E,\alpha}$. This sort of duality is typical in convex analysis, for instance: there is qualitative equivalence between strict convexity of $f$ and the smoothness of the sub-differential map associated to $f$, and vise-versa (see \cite[Theorems 26.1, 26.3]{rockafellar2015convex}), and there is quantitative equivalence of convexity of a Banach space and smoothness of its dual space, and vise-versa (see \cite[Page 193]{xu1991characteristic} and the references therein). Nonetheless, to the best of the authors' knowledge, this is the first time this type of duality has been used to address the quantitative equidistribution properties of flows.

Finally,  the method used here is flexible enough to easily extend Theorem \ref{thm:error_term} to the class of measures with density $f \in W^{1,p}(\T^{2})$ for any $p \ge 1$. This result is of the same style as the extension in \cite{Bec15} to measures with square-integrable density, except, we attain a better bound by looking at a smaller class of measures.

\begin{thm} \label{t:measures} Fix a measure $\mu \ll \lambda_{2}$, such that $d\mu = f d\lambda_{2}$. For each $\alpha \in \R$ and $x=(1,h) \in [0,1]^{2}$, consider
\begin{equation} \label{e:taumu}
	\tau_{\mu,\alpha}(x) \defeq \int_{-1}^{0} f( Y^{x}_{\alpha}(t)) \mathrm{d} t.
\end{equation}
	If there exists $p \in [2, \infty)$ so that $f \in W^{1,p}(\T^{2})$ then for almost every $\alpha \in \R$
	\begin{equation} \label{e:measerror}
	\int_{0}^{T}  f(Y_{\alpha}^{x}(t)) d t - T \mu(E) \le o \left( \log(T)^{1+ \varepsilon} \right),
	\end{equation}
	and for all badly-approximable $\alpha$
	$$
	\int_{0}^{T}  f(Y_{\alpha}^{x}(t)) d t - T \mu(E) \le o \left( \log(T)^{ \frac{1}{p^{\prime}}+ \varepsilon} \right).
	$$
\end{thm}
Note that measures of the form $d \mu = \chi_{E} d \lambda_{2}$ do not satisfy $\chi_{E} \in W^{1,p}(\T^{2})$, so Theorems \ref{thm:error_term} and \ref{t:measures} are entirely distinct.  However, the simpler proof of Theorem \ref{t:measures} illuminates why the bound on the error term in Theorems \ref{thm:error_term} and \ref{t:measures} depend upon the one-dimensional projections of the measure behaving nicely from an analytic point of view. 

In a series of papers, \cite{BDY1, BDYII, BCY3, BCY4}, Beck et al. provide an in-depth discussion of the quantitative behavior of the translation flow on higher genus surfaces, and prominently on square-tiled surfaces of genus $2$. A full summary of these extensive results (almost 500 pages) is beyond the scope of the present article. Nevertheless, 

\underline{Open Question 2:} Does an error bound like \eqref{e:rb} hold for a non-trivial extension of the class of convex sets on higher-genus surfaces?

\section*{Acknowledgment} The authors would like to thank Christoph Aistleitner, J\'{o}zsef Beck, Dmitriy Bilyk, Samantha Fairchild, and Stefan Steinerberger for useful discussions on important aspects of this article. Moreover, we would like to thank the referee for their careful reading.

\section{Preliminaries and proofs of Theorems~\ref{thm:error_term} and \ref{t:measures}}
\label{subsec:background}

In this section we provide background information and then deduce Theorem \ref{thm:error_term} from Theorem \ref{t:sets}. All our techniques prominently use the function on the right edge of the square obtained from measuring the time which the flow in direction $\alpha$ ending at $x$ spends in $E$ within one time unit: More precisely, if $E \subset [0,1]^{2}$ is a measurable set, $\alpha \in \mathbb{R}$, and $x:=(1,h)$, we define the function 
\begin{equation} \label{e:tau}
\tau_{E,\alpha}(h) := \int_{-1}^{0} \chi_E(Y_\alpha^x(t)) \mathrm{d}t.
\end{equation}
It follows from Fubini's Theorem that 
\begin{align*}
\lambda_{2}(E) = \int_0^1 \tau_{E,\alpha}(h) \mathrm{d}h. 
\end{align*}
\subsection{Discretization} The core idea of the proof of Theorem~\ref{thm:error_term} is the following: First we reduce from a continuous problem in $[0,1]^{2}$ to a discrete problem in $[0,1]$ by considering the specific points in time $N \in \mathbb{N}$. If $x=(0,x_{0})$, then at these times the amount of time spent by the flow $Y^{x}_{\alpha}$ in $E$ up to time $N$ is exactly $\sum_{k=1}^N \tau_{E,\alpha}(\left\{k\alpha+x_0\right\})$. So, instead of considering the continuous-time translation flow we consider the discrete-time setting. Indeed, when $N:= \lfloor T \rfloor$, the discretization error satisfies 
$$0 \leq  \int_{0}^T \chi_E(Y(t)) \mathrm{d} t - \int_{0}^N \chi_E(Y(t)) \mathrm{d} t  \leq (T-N) \sqrt{1+\alpha^2} \leq \sqrt{1+\alpha^2},$$
because the error is maximized if $\{Y_{\alpha}^{(0,x_{0})}(s) : s \in (N,T) \} \subset E$. Therefore the discretization error is $O(1)$ justifying our ability to neglect it and only discuss the discrete setting.

The gain from this simplification is that we only need to evaluate $\tau_{E,\alpha}$ for the \textbf{shifted finite Kronecker sequence} $\left\{k\alpha+x_0\right\}_{k=1}^N$ to understand $\int_{0}^{N} \chi_{E}(Y_{\alpha}^{x}(t)) \mathrm{d}t$. Moreover, the technical assumption $\tau_{E,\alpha} \in W^{1,p}(\T)$ allows us to make use of the $L^p$-Koksma Hlawka inequality (Theorem~\ref{t:hick}) in its particularly simple $1$-dimensional setting (Corollary~\ref{c:hick}). This makes our proof both shorter and simpler than the ones in \cite{Bec15} and \cite{GL16} while allowing for a better error term than in \cite{Bec15} and a richer class of sets than in \cite{GL16}.
\subsection{Background on $L^p$-discrepancy} Recall that the \textbf{Sobolev space} $W^{1,p}(\T)$ is defined as the subspace of $\tau \in L^{p}(\T)$, with the additional property that there exists a (suggestively denoted) $\tau^{\prime} \in L^{p}(\T)$, so that
$$
\int_{[0,1]} \tau f^{\prime} = - \int_{0}^{1} f \tau^{\prime} \qquad \forall f \in C^{1}_{c}(\T).
$$
The function $\tau^{\prime}$ is called the \textbf{Sobolev} or \textbf{weak derivative} of $\tau$. On a first read, one could now skip directly to Remark \ref{r:1d} and understand the $1$-dimensional version of the $L^{p}$-Koksma-Hlawka inequality, Corollary~\ref{c:hick}. But, for completeness we introduce the necessary terminology to state the $L^{p}$-Koksma-Hlawka inequality in $\R^{s}$, c.f., \cite{Hic98,Pil20}. For all $\emptyset \neq u \subset S \defeq \{1, \dots, s\}$ and $x \in \R^{s}$ we write $x_{u}$ for the point $(x_{u_{1}},\dots, x_{u_{|u|}})$ and embed $C^{u} = [0,1]^{u}$ naturally in $\R^{s}$. Given any finite collection of points $P \subset [0,1]^{s}$ define $P_{u}$ as the orthogonal projection of $P$ onto $C^{u}$. Moreover, let $[0,x_{u})$ denote the rectangle inside $C^{u}$ whose diagonal corners are the origin and $x_{u}$. Finally, if $P \subset [0,1]^{s}$ is so that $\#|P| < \infty$, we write $\mu_{P} = \frac{1}{|P|} \sum_{x \in P } \delta_{x}$ where $\delta_{x}$ is the Dirac mass centered at $x$. Using this notation we define the $L^{p}$-discrepancy and the $L^{p}$-variation for $1 \leq p \leq \infty$ by:
\begin{defn}[\text{$L^{p}$-Discrepancy}]
For a point set $P \subset I^{s}$ with $|P| = N$ the \textbf{$L^{p}$-discrepancy of $P$} is defined as 
$$
D_{p}^{*}(P) = \max_{u \neq \emptyset} \left\| \left( \frac{\#|P_{u} \cap [0,x_{u})|}{N} - \lambda_{|u|}([0,x_{u})) \right) \right\|_{L^{p}(C^{u})}.
$$ 
\end{defn}
For $p = \infty$ the $L^{\infty}$-discrepancy is just the usual star-discrepany, compare e.g. \cite{KN74}.
\begin{defn}[$L^{p^{\prime}}$-variation]
We define the \textbf{$L^{p^{\prime}}$-variation} of a sufficiently smooth\footnote{For instance, $f \in W^{s,p}([0,1]^{s})$.} function $f : [0,1]^{s} \to \R$ as
$$
V_{p^{\prime}}(f) = \max_{u \neq \emptyset} \left \| \left( \frac{ \partial^{|u|} f}{\partial x_{u}} \bigg|_{x_{S-u}=(1, \dots, 1)} \right)\right\|_{L^{p^{\prime}}(C^{u})}.
$$
\end{defn}
The following result from \cite{Hic98} or more precisely its Corollary~\ref{c:hick} is essential for our proof of Theorem~\ref{thm:error_term}.
\begin{thm}[Hickernell] \label{t:hick} Given a point set $P = \{x_{1}, \dots, x_{N}\}$, we have
$$
\left| \int_{[0,1]^{s}} f d \lambda_{s}- \int_{[0,1]^{s}} f d \mu_{P} \right| \le D_{p}^{*}(P) V_{p^{\prime}}(f)
$$
whenever $p$ and $p^{\prime}$ are H\"older conjugates.
\end{thm}

\begin{rmk} \label{r:1d}
When $s=1$, the definition of $L^{p}$ discrepancy and $L^{p^{\prime}}$-variation simplify dramatically, most noticeably because one need never project onto cubes $C^{u}$, but conveniently one also never needs to consider $|u| > 1$. The latter implies that, if $f \in W^{1,p}(\T)$ then
$$
V_{p^{\prime}}(f) = \| f^{\prime}\|_{L^{p^{\prime}}([0,1])},
$$
is the Sobolev seminorm. On the other hand, the point set $P$ has $L^{p}$-discrepancy
$$
D_{p}^{*}(P) = \left \| \mu_{P}([0, \cdot)) - \cdot  \right\|_{L^{p}([0,1])}.
$$
\end{rmk}

From Remark~\ref{r:1d} we obtain the following

\begin{coro}[$1$-dimensional $L^{p}$-Koksma-Hlawka inequality] \label{c:hick}
If $P \subset [0,1]$ is a collection of $N$ points,and $\mu_{P} = N^{-1} \sum_{x \in P} \delta_{x}$ is the corresponding empirical measure, then
$$
\left| \int_{0}^{1} f(x) dx - \int_{0}^{1} f d\mu_{P} \right| \le D_{p}^{*}(P)\|f^{\prime}\|_{L^{p^{\prime}}([0,1])} \quad \textrm{for all } p,p^{\prime} \in \mathbb{N} \ \textrm{with } \frac{1}{p} + \frac{1}{p^{\prime}} = 1.
$$
\end{coro}

\begin{proof}[Proof of Theorem \ref{thm:error_term}]
In light of Corollary \ref{c:hick} and Corollary \ref{c:sard}, Theorem~\ref{thm:error_term}(A) and (C) follows immediately from Theorem~\ref{t:sets} and the fact that for almost all $\alpha \in \R \setminus \Q$ we have
\begin{equation} \label{e:irrational}
D_{p^\prime}^{*}(P_{N}(\alpha)) \lesssim D^{*}(P_{N}(\alpha)) \lesssim \frac{\log (N)^{1+\varepsilon}}{N},
\end{equation}
compare \cite{KN74}, p.128, while Theorem~\ref{thm:error_term}(B) and (C) follows from Theorem~\ref{t:sets} and Lemma~\ref{l:baa}.
\end{proof}

\begin{lem} \label{l:baa}
If $\alpha$ is badly approximable and $p \ge 2$, then
$$
D^{*}_	{p}(\{ \alpha k \}_{k=1}^{N} ) \lesssim \frac{ \log(N)^{\frac{1}{p^{\prime}}}}{N}.
$$
\end{lem}

\begin{proof}
By \cite[Remark 1]{graham2020irregularity}
$$
\| \mu - \lambda_{1} \|_{W^{-1,p}(\T)} = D_{p}^{*}(\mu,\lambda_{1})
$$
and by \cite[Proposition 4]{graham2020irregularity} for all $p \in [2,\infty]$, there exists $C_{p} < \infty$ so that for all $n \in \N$
\begin{equation} \label{e:fourierbound}
\| \mu - \lambda_{1} \|_{W^{1,-p}(\T)} \le \frac{C_{p}}{n} + \left( \sum_{k=1}^{n-1} \frac{ |\hat{\mu}(k)|^{p^{\prime}}}{k^{p^{\prime}}} \right)^{\frac{1}{p^{\prime}}}.
\end{equation}
We write $\mu_{N} = N^{-1} \sum_{k=1}^{N} \delta_{\{ \alpha k \}}$. According to \cite[Page 16]{steinerberger2021wasserstein} it holds that
$$
| \widehat{\mu_{N}}(k)| \lesssim N^{-1} \{ k \alpha \} 
$$
and $\{k \alpha \}$, the decimal part of $k \alpha$, is approximately $2^{-\ell}$-separated. Using these facts we mimic the computation in \cite{steinerberger2021wasserstein} to obtain
\begin{align*}
\sum_{k=1}^{n-1} \frac{ |\widehat{\mu_{N}}(k)|^{p^{\prime}}}{k^{p^{\prime}}} & \lesssim  \sum_{\ell=1}^{\log(n-1)} \sum_{2^{\ell} \le k \le 2^{\ell+1}} \frac{ |\widehat{\mu_{N}}(k)|^{p^{\prime}}}{k^{p^{\prime}}} \lesssim \sum_{\ell=1}^{\log(n-1)} \sum_{2^{\ell} \le k \le 2^{\ell + 1}} \frac{|N^{-1} \{ k \alpha \}|^{p^{\prime}}  }{k^{p^{\prime}}} \\
& \lesssim \frac{1}{N^{p^{\prime}}} \sum_{\ell=1}^{\log(n-1)} 2^{- \ell p^{\prime}} \sum_{m=1}^{2^{\ell}} \frac{1}{(m/2^{\ell})^{p^{\prime}}} \lesssim \frac{1}{N^{p^{\prime}}} \sum_{\ell=1}^{\log(n-1)} \sum_{m=1}^{\infty} \frac{1}{m^{p^{\prime}}} \\
& \lesssim  \frac{\log(n-1)}{N^{p^{\prime}}}.
\end{align*}
Consequently, \eqref{e:fourierbound} says that for all $n \in \N$, we have $D_{p}^{*}(\mu_{N}) \lesssim \frac{1}{n} + \frac{\log(n-1)^{\frac{1}{p^{\prime}}}}{N}$. Choosing $n = N$ completes the proof.
\end{proof}

It has been informally claimed in the literature, see e.g., the discussion following \cite[Theorem 3]{steinerberger2021wasserstein} that for badly approximable $\alpha$, $D_{p}^{*}(\{ \alpha k \}_{k=1}^{N} ) \lesssim \frac{\log(N)^{\frac{1}{2}}}{N}$ whenever $p \ge 2$. Yet, finding a formal statement and/or proof evaded the authors. But, if this claim were true, Lemma 6.3 is a particularly bad estimate, and Theorem \ref{thm:error_term}(2) can be improved, because $\varepsilon = 0$ would be allowed. Similarly Theorem \ref{thm:error_term}(4) would be improved. Therefore, we ask:

\underline{Open Question 3:} Is the $L^{p}$-discrepancy of $\{ \alpha k\}_{k=1}^{N}$ like $O(\log(N)^{\frac{1}{2}}/N)$ for badly-approximable $\alpha$?

\begin{rmk} The proof of Theorem~\ref{thm:GL} relies on the specific geometric shapes of the sets $E \subset [0,1]^2$ and on approximating the (finite) Kronecker sequences $\left\{ n\alpha \right\}_{n=1}^N$ by a discretized version stemming from the Ostrowski expansion of $N$ with respect to $\alpha$. Readers who are familiar with the proof in \cite{GL16} might wonder whether our estimate of the error term can be significantly improved by using their approximation. In fact, it is not difficult to show that also the star-discrepancy of the discretized Kronecker sequence grows with {at least} $\log(N)N^{-1}$. So, while that technique might yield an improvement, we prefer to present the (in our eyes) much simpler proof here and accept the slightly larger error term.  
\end{rmk}

\begin{rmk} \label{rmk:countable} One might wonder whether the almost everywhere condition in Theorem \ref{thm:error_term} can be sharpened to hold for all but a countable collection of directions as we originally expected. This is untenable as can be seen from the following argument: Let $[a_0,a_1,a_2,\ldots]$ be the continued fraction expansion of $\alpha$ with convergents $(p_i/q_i)_{i \in \mathbb{N}}$. For $q_i < n < q_{i+1}$ we have 
$$\left| n \alpha - (np_i/q_i) \right| < \frac{n}{a_{i+1}q_i}$$
by a basic result from the theory of continued fractions. If $a_{i+1} \gg q_i$, then the irrational rotation is close to {the rational rotation for $\frac{p_{i}}{q_{i}}$} for a long time {$t_{i}$}. Therefore, there exists a rectangle in which the flow spends no time for $t \in [q_i,q_i+t_i]$. If $a_{i+1}^*$ is chosen big enough, then the error term is of order $>T^{1/2}$ implying that a logarithmic error term cannot be obtained. However, for all $a_{i+1} > a_{i+1}^*$ the error term is thus also of order $>T^{1/2}$. Although a corresponding condition needs to be imposed for all $i \in \mathbb{N}$, it still allows for an uncountable number of choices for $\alpha$ with
$$\int_{0}^{T} \chi_{E}(Y^{x}_{\alpha}(t)) dt - T \lambda_{2}(E) > T_n^{1/2}$$
for a sequence $(T_n)_{n \in \mathbb{N}} \to \infty$.
 \end{rmk}

\begin{rmk} The situation described in Remark~\ref{rmk:countable} should be compared to the condition under which Theorem~\ref{thm:GL} was proven in \cite{GL16}. Lemma~2.1 of \cite{GL16} states that for fixed $m > 0$,
	$$\sum_{i=0}^s \frac{a_{i+1}}{q_{i}^{\frac{1}{m}}} \sum_{k=1}^{i+1} a_k$$
is uniformly bounded in $s$ for almost every irrational $\alpha \in (0,1)$. This implies in particular that $\frac{a_{i+1}}{q_i}$ converges to $0$ for almost every irrational $\alpha \in (0,1)$ and prevents that an exploding error like in Remark~\ref{rmk:countable} occurs, where $\frac{a_{i+1}}{q_i}$ grows very fast.
\end{rmk}

We now prove Theorem \ref{t:measures}. Hence, we assume $f \in W^{1,p}(\T^{2})$ and will show this implies $\tau_{\mu,\alpha} \in W^{1,p}(\T)$ for all $\alpha$. 

\begin{proof}[Proof of Theorem \ref{t:measures}]
Let $\lambda_{2}$ denote the Lebesgue measure on $[0,1]^{2}$ and fix some $\mu \ll \lambda_{2}$ an absolutely continuous Borel measure on $[0,1]^{2}$ with density $f$, i.e., $d \mu = f d \lambda_{2}$. By Corollary \ref{c:hick}, equation \eqref{e:irrational}, and Lemma \ref{l:baa} it suffices to show that for all $\alpha \in \R$,  $\tau_{\mu,\alpha} \in W^{1,p}([0,1])$.

Since $f \in W^{1,p}(\T^{2})$, the function $f$ is absolutely continuous with respect to almost every line, see \cite[Chapter 4]{evans2018measure}. We now identify $\T^{2}$ and $\T$ with $[0,1]^{2}$ and $[0,1]$. In particular $f|_{L^{h}}$, where 
\begin{equation} \label{e:Lh}
L^{h} \defeq \{Y^{(1,h)}_{\alpha}(t) : -1 \le t \le 0 \}, 
\end{equation}
is continuous for almost every $h \in I$.  
We claim this implies that difference quotients of  $\tau_{\mu,\alpha}$ converge for almost every $h$. Since $\alpha$ is fixed, we abuse notation and write $Y^{h} = Y^{(1,h)}_{\alpha}$. By dominated convergence theorem and a standard mollification argument
\begin{align*}
\lim_{\delta \downarrow 0} \frac{ \tau_{\mu,\alpha}(h+\delta) - \tau_{\mu,\alpha}(h)}{\delta} & = 
\lim_{\delta \downarrow 0} \int_{0}^{1} \frac{f(Y_{\alpha}^{h+\delta}(t)) - f(Y_{\alpha}^{h}(t))}{\delta} \mathrm{d} t \\
& = \int_{0}^{1} \lim_{\delta \downarrow 0} \frac{f(Y_{\alpha}^{h+\delta}(t)) - f(Y_{\alpha}^{h}(t))}{\delta} \mathrm{d} t ,
\end{align*}
when the final limit exists and is integrable. Since $\{Y_{\alpha}^{h+\delta}(t) - Y_{\alpha}^{h}(t)\} = \delta e_{2}$, it follows that 
$$
\lim_{\delta \downarrow 0}  \frac{f(Y^{h+\delta}(t)) - f(Y^{h}(t))}{\delta} = \partial_{2} f(Y^{h}(t)) \qquad a.e.~ (h,t).
$$

In particular, Fubini's Theorem guarantees that 
$$
\int_{0}^{1} |\tau^{\prime}_{\mu,\alpha}(h)| dh \le \int_{0}^{1} \int_{0}^{1} |\partial_{2}f(Y^{h}(t))| \mathrm{d}t dh = \int_{[0,1]^{2}} | \partial_{2} f| < \infty
$$
verifying $\tau_{\mu,\alpha} \in W^{1,p}([0,1])$. Since $f$ can be extended to $\bar{f} \in W^{1,p}_{loc}(\R^{2})$ by periodicity, it follows that we can again recognize $[0,1]$ as $\T$ and conclude $\tau_{\mu,\alpha} \in W^{1,p}(\T)$. 
\end{proof}
\section{Proof of Theorem~\ref{t:sets}} \label{sec:proof_remainder}
We now return our attention to the more classic case where $\mu = \chi_{E}  d \lambda$ to study sets.  Recall that $\cE_{\sigma}([0,1]^{2}) \subset \cE([0,1]^{2})$ are the algebras generated by $\sigma$-convex\footnote{The exact definition is given in Definition~\ref{def:sigma-convex}.} and convex sets respectively. For $E$ in either class, in order to produce the desired bound on the remainder of $E$ in the direction $\alpha$ we need to show $\tau_{E,\alpha} \in W^{1,p}(\T)$, for some $p \in (1, \infty)$. We first look at some toy examples to build an understanding of what causes $\tau_{E,\alpha}$ to gain or lose regularity.

\begin{example}(Power-growth graphical case) \label{x:ep}
Fix $0 < p^{\prime} < \infty$, $f(x) = c x^{p^{\prime}}$, and 
$$
E_{p^{\prime}} =  \{(x,y) \in [0,1]^{2} : y > f(x) \}.
$$
Then the function $\tau=\tau_{E_{p},0}$ defined as in \eqref{e:tau} is given by
$$
\tau(t) =   \left(\frac{t}{c} \right)^{\frac{1}{p^{\prime}}} = f^{-1}(t).
$$
Consequently, for all $t \neq 0$,
$$
\tau^{\prime}(t)  = \frac{ c^{-1/p^{\prime}}}{p^{\prime}} t^{\frac{1-p^{\prime}}{p^{\prime}}} = (f^{-1})^{\prime}(t).
$$ 
In particular,
\begin{equation*}
\tau \in \begin{cases}
C^{1, \frac{1-p^{\prime}}{p^{\prime}}}([0,1]) & p^{\prime} < 1 \\
\textrm{Lip}([0,1]) & p^{\prime} = 1 \\
C^{\frac{1}{p^{\prime}}}\cap W^{1,s}([0,1]) & \forall s < p, \quad p^{\prime} > 1 \\
\end{cases}
\end{equation*}
\end{example}

The example can be readily verified by observing that since $\alpha = 0$, the function $\tau_{E_{p^{\prime}},0}(t)$ is the length of the level-sets of the function $f(x) = c x^{p^{\prime}}$. See Remark \ref{r:simp} for a more in depth discussion of the simplifications made within this example.  We also compare Example \ref{x:ep} with the ``flat case'' when $p^{\prime}= 0$.

\begin{example}(Horizontal constant case) \label{x:e0} If $E_{0} = \{(x,y) \in [0,1]^{2} : y > c \}$, then $\tau_{0}(t) = \chi_{\{ \cdot \ge c\}}(t)$. In particular, the distributional derivative $\tau_{0}^{\prime} = \delta_{c} \in \cM([0,1])$ is a measure which cannot be recognized as any function whenever $0 < c < 1$.
\end{example}

Examples \ref{x:ep} and \ref{x:e0} provide substantial theoretical insight into regularity of $\tau_{E,\alpha}$. Vaguely, they indicate the regularity of $\tau_{E,\alpha}$ is most threatened at any $h$, where $L^{h}$ (see \eqref{e:Lh}) is tangent to $\partial E$. The danger being, as demonstrated in Example \ref{x:e0}, if $L^{h}$ is tangent to $E$ on a piece where $\partial E$ is flat, then the distributional derivative $\tau_{E,\alpha}^{\prime}$ is a measure, and in particular $\tau_{E,\alpha} \in W^{1,p}(\T)$ cannot be achieved. Nonetheless, Example \ref{x:ep} indicates that tangency of $\partial E$ with $L^{h}$ can be handled whenever $\partial E$ separates from $L^{h}$ fast enough. In fact, example \ref{x:ep} demonstrates that there should be a quantitative relationship between the regularity of $\tau_{E,\alpha}$ and how quickly $\partial E$ separates from $L^{h}$ at points of tangency. The desire to take advantage of this quantitative relationship between $\tau_{E,\alpha}$ and $\partial E$ leads naturally to desiring the boundary decomposition of Lemma~\ref{l:decomp}.

\begin{example}(Linear case) \label{x:el} For non-zero slope $m$, consider $E = \{ (x,y) \in [0,1]^{2} : y > \alpha x \}$. Then, $\tau_{0}^{\prime}(h) = \alpha^{-1}$ for all $h \in (0,1)$.
\end{example}

Example \ref{x:el} shows that flat pieces of $\partial E$ only cause problems when they are tangent to $L^{h}$. We now recall some facts about convex functions. 
\begin{defn}
If $f : \R^{n} \to \R$ is convex, a direction $x^{*}$ is called a \textbf{sub-gradient to $f$ at $x$} if,
$$
f(y) \ge f(x) + \langle x^{*}, y-x \rangle \qquad \forall y \in \R^{n}.
$$
We define the multi-valued map $\partial f(x) = \{ x^{*} : x^{*} \text{is a sub-gradient to $f$ at $x$}\}$. The map $\partial f(x)$ is called the \textbf{sub-differential of $f$}. When $f$ is differentiable at $x$, it follows $\partial f(x)= \{\nabla f(x)\}$.
\end{defn}
The following proposition is well-known. For instance, \eqref{e:mono} can be found in \cite[Chapter 24]{rockafellar2015convex} while the moreover statement, due to Alexandrov, can be found in \cite[6.4 Theorem 1]{evans2018measure}.
\begin{prop} \label{p:cx}
A function $f : \R^{n} \to \R$ is \textbf{convex} if and only if for all $x,y \in \R^{n}$, and all $x^{*} \in \partial f(x), y^{*} \in \partial f(y)$ we have
\begin{equation} \label{e:mono}
\langle x^{*} - y^{*}, x-y \rangle \ge 0.
\end{equation}

Moreover, if $f$ is convex, there exists a set $U \subset \rn$ so that $f \in C^{2}(U)$ and $\lambda^{n}( \rn \setminus U)=0$.
\end{prop}
We also often abuse notation and simply write $\partial f(x)$ instead of writing both $x^{*}$ and for all $x^{*} \in \partial f(x)$. For example, \eqref{e:mono} would be written $\langle \partial f(x) - \partial f(y), x-y \rangle \ge 0$ for all $x,y \in \R^{n}$.

\begin{defn} \label{def:sigma-convex}
If, for some $\sigma \in [2, \infty)$ and $c > 0$, the function $f: \R^{n} \to \R$ is uniformly convex with
\begin{equation} \label{e:sigma}
\langle\partial f(x) - \partial f(y), x-y \rangle \ge c |x-y|^{\sigma},
\end{equation}
then $f$ is called \textbf{$\sigma$-convex with constant $c$}. When the precise constant $c$ is not particularly important, we just say $f$ is $\sigma$-convex.
\end{defn}

Convexity is a global property of functions. Nonetheless, convexity also implies a lot of local structure and vise-versa. A $\sigma$-convex set is a set which is locally the graph of a $\sigma$-convex function at every point in the boundary. Roughly, $\sigma$-convexity of a set means that at every point in the boundary, the boundary of the set separates from the tangent plane at least as quickly as $c |x|^{\sigma}$, or in other words, that the boundary is never too flat. If $p > 2$, then the set $\{ x \in \R^{n} : \| x \|_{\ell^{p}} \le 1 \}$ is an example of a $p$-convex set which is not $s$ convex for any $2 \le s < p$.

Regarding the regularity of $\tau_{E,\alpha}$, Examples \ref{x:ep} - \ref{x:el} indicate restrictions on the global geometry, e.g., convexity, of $E$ may not be strictly necessary. Instead we focus on requiring that the boundary separates from tangent planes sufficiently fast (or, when this fails, that the boundary is not in the direction of the flow $Y^{x}_{\alpha}(t)$). 
As preparation for the proof of Theorem~\ref{t:sets}, we collect some important properties of $\sigma$-convex functions in a separate Proposition.
\begin{prop}[Properties of $\sigma$-convex functions] \label{p:scxfunc}
If $f : \R \to \R$ is $\sigma$-convex with constant $c$, then:
\begin{enumerate}
\item[(A)] The right derivative $f^{\prime}_{+}(x) = \lim_{h \downarrow 0} \frac{ f(x+h) - f(x)}{h}$ is an increasing right-continuous function, the left derivative $f^{\prime}_{-}(x)$ is an increasing left-continuous function.
\item[(B)] For all $x,y \in \mathbb{R}$, the inequality $|\partial f (x) - \partial f(y)| \ge c |x-y|^{\sigma-1}$ holds.
\item[(C)] If $0 \in \partial f(x_{0})$ then $g = (f|_{[x_{0},\infty)})^{-1}$ is well-defined. Moreover, for all $h > x_{0}$,
\begin{equation} \label{e:gbound}
\limsup_{\delta \to 0} \left| \frac{ g(h+\delta) - g(h)}{\delta} \right| \le c^{1 - \frac{1}{\sigma}} \sigma |h-x_{0}|^{-(1-\frac{1}{\sigma})}.
\end{equation}
\end{enumerate}
\end{prop}

\begin{proof}
(1) is known to hold for any convex function, see for instance \cite[Chapter 24]{rockafellar2015convex}. (2) follows by applying Cauchy-Schwarz on the left side of \eqref{e:sigma} then dividing by $|x-y|$. For (3), since $f$ is $\sigma$-convex, $x > y$ implies $\partial f(x) > \partial f(y)$. Since $0 \in \partial f(x_{0})$ it follows $f|_{[x_{0},\infty)}$ is strictly increasing and hence has an inverse. So, it only remains to show \eqref{e:gbound}.

Without loss of generality, $x_{0} = 0$. The general result can be recovered by translation. Define $\psi(x) = c \frac{|x|^{\sigma}}{\sigma}$. Since $f$ is $\sigma$-convex and $0 \in \partial f(0)$, it follows $f - \psi$ is a convex function and $0 \in \partial \left( f - \psi \right)$. In particular, $f(x) \ge \psi(x)$ and $f^{\prime}(x) \ge \psi^{\prime}(x)$. Consequently, if $\phi= (\psi|_{\R^{+}})^{-1}$ then for all $h > 0$, $g(h) \le \phi (h), g^{\prime}(h) \le \phi^{\prime}(h)$.

Evidently $\phi(x) = (\sigma c^{-1})^{\frac{1}{\sigma}} x^{\frac{1}{\sigma}}$ and $\phi^{\prime}(x) = c^{\frac{-1}{\sigma}} (\sigma x)^{\frac{-1}{\sigma^{\prime}}}$. Hence, if $g^{\prime}(h)$ is well-defined,
$$
\lim_{\delta \to 0} \frac{g(h+\delta) - g(h)}{\delta} = g^{\prime}(h) \le \phi^{\prime}(h) = c^{-\frac{1}{\sigma}} (\sigma x)^{-\frac{1}{\sigma^{\prime}}}.
$$
Since $-g$ is a convex function, it follows $g^{\prime}_{-}(x) \ge \partial g(x)$ and $g^{\prime}_{-}$ is left-continuous and increasing, by part (1).
Since $g$ is differentiable almost everywhere, choose $h_{k} \uparrow h$ so that $g^{\prime}(h_{k})$ exists for each $k$. Then, by applying the differentiable case to this sequence, we achieve \eqref{e:gbound} by observing 
$$
g^{\prime}(h_{k}) =g^{\prime}_{-}(h_{k}) \to g_{-}^{\prime}(h) \ge \partial g(h).$$
\end{proof}

We denote the algebra generated by convex ($\sigma$-convex resp.) sets in $[0,1]^{2}$ by $\cE([0,1]^{2})$ ($\cE_{\sigma}([0,1]^{2})$ resp.), and as usual we call it the {\bf algebra of convex ($\sigma$-convex) sets} respectively.

Since a convex set $E_{0}$ is locally the epigraph of a convex function, it natural to associate to $x \in \partial E_{0}$, some special line $L_{x}$ and a function $\psi_{x}$, not necessarily unique, for which near $x$, {the boundary} $\partial E_{0}$ is the graph of $\psi_{x}$ over the line $L_{x}$ with some appropriate choice of orientation. Morally, the same should be true for any $E \in \cE([0,1]^{2})$, but one has to be careful about places where the boundaries of the various sets from which $E$ was constructed intersect. We handle these complications by introducing convexity and concavity of boundary pieces for sets in $\cE([0,1]^{2})$.

\begin{defn}[Convexity and concavity of boundary pieces] \label{d:bp}
	Fix a measurable $E \subset \R^{2}$, and a closed connected $C \subset \partial E$. Then $C$ is called a \textbf{convex piece of $\partial E$} if: 
	
	\begin{enumerate} \item For all $x \in C$, there exists a line $L_{x}$ containing $x$ and some $\varepsilon > 0$ so that $C \cap B(x,\varepsilon)$ is the graph of a convex function $\psi$ defined on a subset of $L_{x}$, and 
	\item the convex hull of $C$ is a subset of $E$.
	\end{enumerate}
	
	We say $C \subset\partial E$ is a \textbf{concave piece of $\partial E$} if $C$ is a convex piece of $\partial(E^{c})$, the complement of $E$. 
	
	Given $\sigma \in [2,\infty)$ we say $C \subset \partial E$ is a \textbf{$\sigma$-convex (with constant $c$) piece of $\partial E$} if $C$ is a convex piece of $\partial E$ and the convex function $\psi$ is $\sigma$-convex. Similarly a \textbf{$\sigma$-concave piece of $\partial E$} is a $\sigma$-convex piece of $\partial (E^{c})$ for $E^{c}$.
\end{defn}

The following lemma is now evident
\begin{lem} \label{l:decomp}
If $E \in \cE([0,1]^{2})$ then there exists some finite collection of boundary pieces, $\cC$, so that $\partial E  = \cup_{C \in \cC}$ and each $C \in \cC$ is either a convex or concave boundary piece.

For $\sigma \ge 2$, if $E \in \cE_{\sigma}([0,1]^{2})$ then there exists some finite collection of boundary pieces so that $\partial E = \cup_{C \in \cC}$ and each $C \in \cC$ is either a $\sigma$-convex or $\sigma$-concave boundary piece.
\end{lem}

\begin{defn} \label{d:degenerate}
Fix $\sigma \in [2,\infty)$, and any $E \in \cE([0,1]^2)$ with $\cC$ as in Lemma \ref{l:decomp}. 

For each $C \in \cC$, there exists directions $\theta_{C}, \theta_{C}^{\perp}$, a line $L_{C}$ an interval $D \subset L_{C}$ and a convex function $\psi_{C} : D  \to \R$  so that $C = \{ x + \theta_{C}^{\perp} \psi_{C}(x) : x \in D \}$. We define the set of {\textbf{$\sigma$-degenerate slopes of $\psi$}} from the left or right by
\begin{equation*}
\cD_{\sigma}^{\pm}(\psi) = \left\{ m = \psi^{\prime}_{\pm}(x) ~ \big| \lim_{\pm(y-x) \downarrow 0} \left \langle m - \partial f (y) , x-y \right \rangle = o(|x-y|^{\sigma}) \right\},
\end{equation*}
 and the set of \textbf{$\sigma$-degenerate slopes of $\psi$} by
$$
\cD_{\sigma}(\psi) =  \cD_{\sigma}^{+}(\psi) \cup \cD_{\sigma}^{-}(\psi).
$$

For each boundary piece $C \in \cC$ define the set of {\textbf{$\sigma$-degenerate directions of $C$}} as 
\begin{equation*}
\cA_{\sigma}(C) = \left\{ \theta = \frac{ \theta_{C} + m \theta_{C}^{\perp}}{|\theta_{C} + m \theta_{C}^{\perp}|} \bigg| m \in \cD_{\sigma}(\psi) \right\}.
\end{equation*}
This gives rise to the set of {\textbf{$\sigma$-degenerate directions of $E$}}
$$
\cA_{\sigma}(E) = \bigcup_{C \in \cC}\cA_{\sigma}(C)
$$
and the corresponding set of {\textbf{$\sigma$-degenerate slopes of $E$}}
$$
\cD_{\sigma}(E) = \bigcup_{C \in \cC}  \cD_{\sigma}(\psi_{C}) .
$$

\end{defn}

We now prove Proposition \ref{p:sard} and Corollary \ref{c:sard} which will guarantee that for any $\sigma \in [2,\infty)$ and any $E \in \cE([0,1]^2)$ the set of $\sigma$-degenerate directions of $E$ is countable.

\begin{prop} \label{p:sard}
For any convex $f : \R \to \R$, {the set of $2$-degenerate slopes} $\cD_{2}(f)$ is countable.
\end{prop}

This proposition is analogous to Sard's theorem applied to the one-sided derivative of a convex function. Indeed, if $f$ is a $C^{2}$ convex function and $x$ is so that $\langle m - \partial f(y), x-y) \rangle = o(|x-y|^{2})$ for $y$ near $x$, then $\partial f(x) = f^{\prime}(x) = m$ and $f^{\prime \prime}(x) = 0$. So, the sets $\cD_{2}^{\pm}(f)$ are like image of the critical set of the one sided derivatives of $f$.

\begin{proof}
Let $f : \R \to \R$ be convex. For each choice of sign $\pm$ we observe we can write
$$
\cD_{2}^{\pm}(f) = \left\{ m = f^{\prime}_{\pm}(x)  \mid \left\langle m - \partial f(y), \frac{x-y}{|x-y|} \right\rangle \le |x-y| \omega_{x}(|x-y|) \quad \forall ~ y > x \right\},
$$
for some gauge function $\omega_{x}$, namely $\omega_{x} : [0, \infty) \to [0,\infty)$ is non-decreasing and continuous with $\omega_{x}(0) = 0$. We will show that $\cD_{2}^{\pm}(f)$ are each countable. Consider $m_{1}, m_{2} \in \cD_{2}^{+}(f)$ so that, without loss of generality $f^{\prime}_{+}(x_{i}) = m_{i}$ implies $x_{2} > x_{1}$. By convexity of $f$ and $m_{1} \in \cD_{2}^{+}(f)$, it follows
$$
0 \le  \left\langle m - \partial f(y), \frac{x_{1}-y}{|x_{1}-y|} \right\rangle = |x_{1}-y| \omega_{x_{1}}(|x_{1}-y|) \quad \forall y > x_{1}.
$$
Equivalently,
$$
|m - \partial f(y)| = |x_{1}-y| \omega_{x_{1}}(|x_{1}-y|).
$$
Choose $\delta$ small enough that $|x-y| \omega(|x-y|) < m_{2} - m_{1}$ whenever $0  < y-x < \delta$. Then $0 < y - x < \delta$ implies $f^{\prime}_{+}(y) < m_{2}$. In particular, if $I_{i} = \{ x : f^{\prime}_{+}(x) = m_{i}\}$ there is a positive distance between $I_{1}$ and $I_{2}$. Thus, $\cD^{+}_{2}(f)$ (and similarly $\cD^{-}_{2}(f)$) are countable.
\end{proof}

\begin{coro} \label{c:sard}
If $\sigma \in [2,\infty)$, $\Omega \subset \R^{2}$, and $E \in \cE([0,1]^2)$ then $\cA_{\sigma}(E)$, the set of $\sigma$-degenerate directions of $E$, is countable.
\end{coro}

The corollary follows immediately from the finiteness of $\cC$ in Lemma \ref{l:decomp}, the definition of degenerate direction (Definition \ref{d:degenerate}), Proposition \ref{p:sard} and the observation that if $\sigma < \sigma^{\prime}$ then $D_{\sigma^{\prime}}(f) \subset D_{\sigma}(f)$.

To prepare for the proof of Theorem \ref{t:sets} we provide a more in-depth discussion of Examples \ref{x:ep}-\ref{x:e0} and how they can be seen as models for the more general setting with $E \in \cE_{\sigma}([0,1]^{2})$.

\begin{rmk} \label{r:simp} We start by discussing the dramatic simplifications.

1) As $f$ is convex, we know that $\{ f = t\}$ has $0$, $1$, or $2$ points. Example \ref{x:ep} conveniently has $\{f =t\}$ is always a set with $1$ point. This is the first key observation for why $\tau_{E,0}$ acts identically to $f^{-1}$ in this example.

2) The domain of $f$ is tangent to the flow $Y^{x}_{0}(t)$. This is what allows $\tau_{E,0}$ to be recognized as the length of the level sets of $f$, and not the length of the intersection of $\epi(f)$ with the graph of $y = m_{\alpha} x$, where $m_{\alpha}$ represents the slope of the line in the direction $\alpha$, as a graph over the same domain as $f$, see \eqref{e:malpha}. By the Pythagorean Theorem writing $F(x) = f(x) - m_{\alpha}x$ the length of this intersection is $\sqrt{1+|m_{\alpha}|^{2}} \lambda_{1}(F^{-1}(\cdot))$. 

3) The domain of $f$ is perpendicular to the direction $e_{2}$. The direction $e_{2}$ is special for the function $\tau_{E,\alpha}$ since $e_{2}$ is the direction the flow $Y^{(1,h)}_{\alpha}(t)$ is translated as $h$ changes. Therefore, this simplification allows $\tau_{E,\alpha}(h)$ to correspond to the $h$-level set of $F$ instead of $\tau_{E,\alpha}(h+\delta)$ and $\tau_{E,\alpha}(h)$ corresponding to two level sets of $\{F = y_{1}\}$ and $\{F=y_{2}\}$ for some $|y_{1} - y_{2}|= \frac{\delta}{ (e_{2} \cdot \theta_{L}^{\perp})}$. This means when $\theta_{L}$ is not perpendicular to $e_{2}$, and while avoiding dividing by zero, one could write $G(e_{2} \cdot \theta_{L}^{\perp} h) = F(h)$ and up to translation identify 
\begin{equation*}
\tau_{E,\alpha}(h) = \sqrt{1+|m_{\alpha}|^{2}} \lambda_{1} \left( \left\{ F = \frac{h}{e_{2} \cdot \theta_{L}^{\perp}} + y_{0} \right\} \right),
\end{equation*} 
for some vertical translation $y_{0}$ of the graph of $G$. Therefore, if $\partial E$ is the graph of a $C^{1}$-bijection $f$ over the line $L$, then 
\begin{equation} \label{e:graphcase}
\tau_{E,\alpha}^{\prime}(h) = \frac{\sqrt{1+|m_{\alpha}|^{2}}}{e_{2} \cdot \theta_{L}^{\perp}} (G^{-1})^{\prime}(h+y_{0}).
\end{equation}
In particular, if the slope of $G$ can be bounded away from zero for
some range of $h$, then on that interval, $\tau_{E,\alpha}$ will be Lipschitz. 
\end{rmk}

If $f$ were not injective, formally, the above analysis still indicates

\begin{equation} \label{e:cancellation}
\tau_{E,\alpha}^{\prime}(h) = \sum_{x : f(x) - m_{\alpha}(x) = c_{h}} \frac{\sqrt{1 + |m_{\alpha}|^{2}}}{(\theta_{L}^{\perp} \cdot e_{2}) M(f^{\prime}(x),\alpha)}
\end{equation}
where $M(f^{\prime}(x),\alpha)$  is the slope of $L^{h}$ with respect to the tangent line to $f$ at $x$. That is,
$$
M(f^{\prime}(x),\alpha) = \frac{(e_{1} + \alpha e_{2}) \cdot (\theta_{L} + f^{\prime}(x) \theta_{L}^{\perp})}{(e_{1} + \alpha e_{2}) \cdot (- f^{\prime}(x) \theta_{L} + \theta_{L}^{\perp})}.
$$

In theory, cancellation can occur in \eqref{e:cancellation}, allowing very specific sets and directions to satisfy $\tau_{E,\alpha} \in W^{1,s}$ despite the hypotheses of Theorem \ref{t:sets} failing, see Figure \ref{f:cancellation}. The loss in the proof arises from putting absolute values inside the summation of \eqref{e:cancellation}, preventing multiple singular parts of $\tau_{E,\alpha}^{\prime}$ from potentially canceling out. 
\begin{figure} 
\begin{tikzpicture}[scale=1.2]

\begin{scope}[rotate=15]

\fill [lightgray, domain=-1:1, variable=\x]
(-1, 0)
-- plot ({\x}, {-(sqrt(1-\x*\x))})
-- (1, 0)
-- cycle;

\fill [lightgray, domain=-2:2, variable=\x]
(-1, 0)
-- plot ({\x}, {(2*sqrt(1-1/4*\x*\x))})
-- (1, 0)
-- cycle;

\fill [white, domain=-1:1, variable=\x]
(-1, 0)
-- plot ({\x}, {(sqrt(1-\x*\x))})
-- (1, 0)
-- cycle;

\draw [thick] [->] (-2,0)--(2,0);

\draw [domain=-1:1, variable=\x]
plot ({\x}, {-sqrt(1 - \x*\x)	)});

\draw [domain=-1:1, variable=\x]
plot ({\x}, {sqrt(1 - \x*\x)	)}); 

\draw [domain=-2:2, variable=\x]
plot ({\x}, {2*sqrt(1 - 1/4*\x*\x)	)});  

\draw [thick,blue] [->] (1.75,0.25)--(2.25,0.25) node[below, right] {$\vec{\alpha}$};

\draw [thick,red] [<->] (-1,0.2)--(1,0.2) node[above] at (0,0.2) {$r/2$};
\draw [thick,red] [<->] (-2,-0.2)--(-1,-0.2) node[below] at (-1.5,-0.2) {$r/4$};
\draw [thick,red] [<->] (1,-0.2)--(2,-0.2) node[below] at (1.5,-0.2) {$r/4$};

\end{scope}

\draw [thick] (-3,-3)--(3,-3);
\draw [thick] (3,-3)--(3,3);
\draw [thick] (3,3)--(-3,3);
\draw [thick] (-3,3)--(-3,-3);

\node[above,right] at (2,2) {$I^2$};

\end{tikzpicture}
\caption{ \label{f:cancellation}
An example of a set $E$ and direction $\alpha$ so that, despite the existence of flat pieces of the boundary in the direction $\alpha$, the function $\tau_{E,\alpha}$ would behave nicely. The sort of cancellation occurring here is lost in our proof, because we ultimately put absolute values on each term inside the summation in \eqref{e:cancellation}.
}
\end{figure}
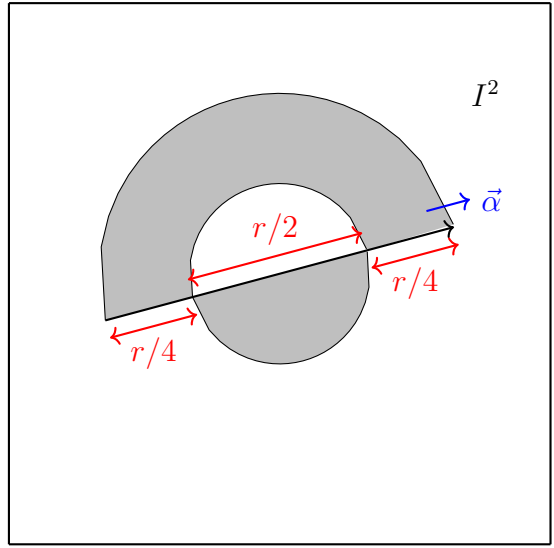

We now set some terminology and notation that will be used to maintain succinctness of exposition without losing precision.

Throughout the proof we will be discussing $x_{0} \in \partial E$. Since $E \in \cE([0,1]^{2})$ there are $C \in \cC$ so that $x \in C$ and $C$ is the graph of some $\psi$ over some $L_{C} \ni x_{0}$. We will let $\theta_{C}$ denote a choice of the positive direction of the line $L_{C}$. We choose $\theta_{C}^{\perp}$ to denote the \textbf{positive ``vertical'' direction}. Since $L_{C}$ is a supporting hyperplane of $C$, this is equivalent to choosing the direction so that $\langle y-x_{0}, \theta_{C}^{\perp} \rangle \ge 0$ for all $y \in C$ when $C$ is a convex piece of $\partial E$ and $\langle y - x_{0}, \theta_{C}^{\perp} \rangle \le 0$ for all $y \in C$ when $C$ is a concave piece of $\partial E$. 

When we talk about $\psi$ being a function over the line $L_{C}$, we keep $L_{C}$ embedded in $[0,1]^{2}$. Then the graph of $\psi$ means the set of points $\{ x + \psi(x) \theta_{C}^{\perp} : x \in L_{C} \cap \textrm{dom}(\psi) \}$ and the set of $x \in L_{C}$ is parametrized by $\{x_{0} + t \theta_{C} \}$. We say $x_{0} + t \theta_{C}$ is left or right of $x_{0}$ depending on whether $t < 0$ or $t > 0$ respectively.

Once we fix $L_{C}, \theta_{C}$ and $\theta_{C}^{\perp}$ we also define $m_{\alpha} = m_{\alpha}(C)$ by 
\begin{equation} \label{e:malpha}
m_{\alpha} = \frac{\theta_{C}^{\perp} \cdot (e_{1} + \alpha e_{2})}{\theta_{C} \cdot (e_{1} + \alpha e_{2})},
\end{equation}
whenever the denominator is non-zero. In this case, if $h_{0}$ is so that $x_{0} \in L^{h_{0}} \cap \partial E$, then the graph over $L_{C}$ of the function $f(x_{0} + t \theta_{C}) = t m_{\alpha}$ is contained in $L^{h_{0}}$, where $L^{h_{0}}$ is as in \eqref{e:Lh}.
Moreover, $L^{h_{0}+\delta}$ is contained in the graph of $f_{\delta}(x_{0} + t \theta_{x_{0}}) = t m_{\alpha} + \frac{\delta}{ (e_{2} \cdot \theta_{C}^{\perp})}$. Indeed,
\begin{align} \label{e:translate}
e_{2} \cdot \left[ f_{\delta}(x_{0} + t \theta_{x_{0}} ) - f(x_{0} + t \theta_{x_{0}}) \right] = e_{2} \cdot \left[ \frac{\delta}{ (e_{2} \cdot \theta_{x_{0}}^{\perp})} (\theta_{C}^{\perp}) \right] = \delta.
\end{align}

We recall 
\begin{equation} \label{e:thetaalpha}
\theta_{\alpha}  = \frac{e_{1} + \alpha e_{2}}{\sqrt{1 + |\alpha|^{2}}}
\end{equation}
is the normalized direction of our trajectory.

\begin{proof}[Proof of Theorem~\ref{t:sets}]
Fix $E \in \cE([0,1]^{2})$, $h_{0} \in [0,1]$, $\sigma \in [2,\infty)$, $\alpha \not \in \cD_{\sigma}(E)$, and $x_{0} \in L^{h_{0}} \cap \partial E$. Choose some $C \in \cC$ with $x_{0} \in C$. Then $C$ is the graph over some line $L_{C}$ by some convex (or concave) $\psi$. Let $H_{C}$ denote the collection of $h \in [0,1]$ so that $L^{h} \cap C \neq \emptyset$. Then $H_{C}$ is an interval and $C \cap \partial E$ contains $1$ or $2$ points for every $h \in H_{C}$. We write $C \cap \partial E = \{ x_{\pm}(h)\}$ where $x_{+}(h) \cdot \theta_{C} \ge x_{-}(h) \cdot \theta_{C}$. Since $C$ is the graph of a convex function, it turns out that $h \mapsto x_{\pm}(h)$ are continuous functions on their domain of definition.

Define $f(t) = \psi(x + t \theta_{x_{0}})$. By assumption $f$ is convex or concave. Without loss of generality, suppose $f$ is convex.  We split into several cases depending on $C$ and $\psi$. 

\vspace{.25em} \underline{Case 0:} $\theta_{C}^{\perp} \cdot e_{2} = 0$.

In this case, the discussion within Remark \ref{r:simp} does not apply, because $L^{h}$ and $L^{\tilde{h}}$ are translations of each other in the $\theta_{C}$-direction. In fact, this makes our job easier. Since $C$ is the graph of $\psi$, for all $h,\tilde{h} \in H_{C}$ the inequality
\begin{align*}
\left| \theta_{\alpha} \cdot (x_{\pm}(h) - x_{\pm}(\tilde{h}) \right|^{2} & = \left|\theta_{\alpha} \cdot \left( \left(h - \tilde{h} \right) \theta_{C} + \left(  \psi(x) - \psi(\tilde{x}) \right) \theta_{C}^{\perp} \right) \right|^{2} \\
& \le |h - \tilde{h}|^{2} + |\psi(x) - \psi(\tilde{x})|^{2}
\end{align*}
holds, where $x,\tilde{x}$ are the points in $L_{C}$ with heights $h,\tilde{h}$ respectively. In particular $|x-\tilde{x}| = |h-\tilde{h}|$. Since $\psi$ is convex, it is locally Lipschitz. Thus, there exists $\delta(h) > 0$ small enough that $h,\tilde{h} \in B_{\delta}(h_{0})$ ensures $|\psi(x) - \psi(\tilde{x})| \le L |h - \tilde{h}|$ for some (local) Lipschitz constant $L$. In particular, $\left| \theta_{\alpha} \cdot (x_{\pm}(h) - x_{\pm}(\tilde{h}) \right| \le \sqrt{ 1 + L^{2}} |h-\tilde{h}|$ for all $h \in I_{C}(x_{0}) \defeq (h_{0} - \delta(h), h_{0} + \delta(h))$.

\underline{Case 1:} $f^{\prime}_{-}(0) < m_{\alpha}$ or $m_{\alpha} < f^{\prime}_{+}(0)$ and $e_{2} \cdot \theta_{C}^{\perp} \neq 0$.

\begin{figure} 
\begin{tikzpicture}
\draw [thick] (-3,0)--(3,0) node[right, below] {$L_{x_0}$};

\node at (0,0) [circle,fill,inner sep=1.5pt]{};
\node[above,right] at (2,1) {$C = \textrm{graph}(f)$};

\draw [domain=-2.5:2.5, variable=\x, red] [<->]
plot ({\x}, { 1/2 * \x*\x	});
\draw [thick,blue] [->] (-2,1.875)--(-1,0.375) node[left] at (-1.2,0.5) {$\vec{\alpha}$};
\draw [thick,blue] [->] (-0.5,0.875)--(0.5,-0.875) node[left, right] {$\vec{\alpha}$};
\draw (0,0) circle node[below] {$x_0$};
\end{tikzpicture}
\caption{ \label{f:2.1and2.2}
To the right of $x_{0}$ we see Case 1. As one moves right, the slope of $f$ only gets further from the slope of $\alpha$. So, the distance between $m_{\alpha}$ and $f_{+}^{\prime}(0)$ easily bounds the distance between $\partial f$ and $m_{\alpha}$ as one moves right. \ \newline
To the left of $x_{0}$ we see Case 2. A priori, we have no idea how far we can go before the slope of the graph agrees with $m_{\alpha}$. Nonetheless, continuity of $f_{-}^{\prime}$ from the left ensures there is some distance one can move left before this happens. 
}
\end{figure}
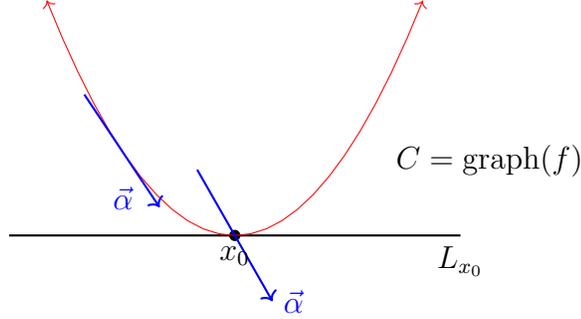

The main idea for this case is to use monotonicity of $\partial \psi$ and hence of $\partial f$ to demonstrate that the accessible tangent directions for $C$, i.e., the directions $(1, f^{\prime}_{\pm}(t))$ remain far from the direction $\theta_{\alpha}$. Consequently, the epigraph of $\psi$, denoted $\epi(\psi)$, is contained in some closed cone, and the direction $\theta_{\alpha}$ ``lies below'' this cone. A bound on the $\alpha$-directional movement of the boundary points, similar to \eqref{e:graphcase}, will be given, where $G^{\prime}$ will be bounded uniformly below. See Figures \ref{f:2.1and2.2}, \ref{f:2.1and2.3}.

By convexity of $f$, it follows $f_{+}^{\prime}(0) \le \partial f(t)$ for all $t > 0$ and therefore
\begin{equation} \label{e:c1gap}
\partial f (t) - m_{\alpha} \ge f^{\prime}_{+}(0) - m_{\alpha} > 0 \qquad \forall t \ge 0.
\end{equation}. 

For $h,\tilde{h} \in H_{C}$ we use the graphicality of $\psi$ to define $x, \tilde{x}$ by 
\begin{equation} \label{e:graphrep}
x_{+}(h) = x \theta_{C} +\psi(x) \theta_{C}^{\perp} \quad \text{ and } \quad x_{+}(\tilde{h}) = \tilde{x} \theta_{C} + \psi(x) \theta_{C}^{\perp},
\end{equation}
and without loss of generality suppose $x_{0} \cdot \theta_{C} \le x  < \tilde{x} $.  As in Remark \ref{r:simp}(3) or \eqref{e:translate} we know $(\psi(x) - \psi(\tilde{x}) - m_{\alpha}(x-\tilde{x})) = \frac{h-\tilde{h}}{\theta_{C}^{\perp} \cdot e_{2}}$.  Then, from \eqref{e:graphrep} it follows
\begin{align}
\nonumber \bigg| & \theta_{\alpha} \cdot \left(\frac{x_{+}(h) - x_{+}(\tilde{h})}{h-\tilde{h}} \right) \bigg|^{2}  = \left| \theta_{\alpha} \cdot \left( \theta_{C}  \frac{x-\tilde{x}}{h-\tilde{h}} + \theta_{C}^{\perp} \frac{ \psi(x) - \psi(\tilde{x})}{h-\tilde{h}}  \right) \right|^{2}  \\
\nonumber & \le \left| \frac{1}{ \theta_{C}^{\perp} \cdot e_{2}} \right|^{2} \left[ \left| \frac{x- \tilde{x}}{\psi(x) - \psi(\tilde{x}) - m_{\alpha}(x-\tilde{x})} \right|^{2} + \left| \frac{\psi(x) - \psi(\tilde{x})}{\psi(x) - \psi(\tilde{x}) - m_{\alpha}(x-\tilde{x})} \right|^{2} \right]\\
\label{e:lipboundcomp}  & =  \left| \frac{1}{ \theta_{C}^{\perp} \cdot e_{2}} \right|^{2}  \left[\left| \frac{1}{ \frac{\psi(x) - \psi(\tilde{x})}{x-\tilde{x}} - m_{\alpha}} \right|^{2}  + \left| \frac{ \psi(x) - \psi(\tilde{x})}{\frac{ \psi(x) - \psi(\tilde{x})}{x-\tilde{x}}  - m_{\alpha}} \right|^{2} \right] \\
\nonumber & \le \frac{c_{0}}{|\theta_{C}^{\perp} \cdot e_{2}|^{2}} \frac{1}{|f^{\prime}_{+}(0) - m_{\alpha}|^{2}},
\end{align}
where the final line used the convexity and boundedness of $\psi$ in addition to \eqref{e:c1gap}. Note $c_{0} = 9 \ge 1 + 4 \| \psi \|_{L^{\infty}}^{2}$ suffices. In particular, we have proven a Lipschitz bound of the form
\begin{equation} \label{e:lipbound}
\left| \theta_{\alpha} \cdot (x_{+}(h) - x_{+}(\tilde{h})) \right| \le c_{1} |h - \tilde{h}| \quad \forall h,\tilde{h} \in H_{C}.
\end{equation}
Therefore, the bound \eqref{e:lipbound} holds throughout $H_{C}$, which either contains an open interval around $h_{0}$, or a half-open interval around $h_{0}$. In the latter case, we can symmetrize the interval around $h_{0}$ by replacing the natural bound of $0$ on the displacement when $x_{+}(h)$ does not exist, by the bound \eqref{e:lipbound} so that the bound \eqref{e:lipbound} holds on some open interval, denoted $I_{C}(x_{0})$, containing $h_{0}$.

When $f_{-}^{\prime}(0) < m_{\alpha}$ then $\theta_{\alpha} \cdot (x_{-}(h) - x_{-}(\tilde{h}))$ can be bounded in the same fashion.

\underline{Case 2:} $m_{\alpha} < f^{\prime}_{-}(0)$ or $f^{\prime}_{+}(0) < m_{\alpha}$ and $e_{2} \cdot \theta_{C}^{\perp} \neq 0$.

The main idea for this case is to use the continuity of $f^{\prime}_{+}(t)$ from the right and of $f^{\prime}_{-}(t)$ from the left to show that sufficiently near to $t=0$ the absolute value $|\partial f(t)|$ remains bounded above by $|m_{\alpha}|$. So, $\epi( \{y = m_{\alpha} x\})$ is contained in a cone which stays a positive distance away from $\partial f(t)$ for $t$ small enough, producing the desired bound, analogously to the previous case, see Figure \ref{f:2.1and2.2}.

We will first consider the instance $f^{\prime}_{+}(0) < m_{\alpha}$. 

Since $f^{\prime}_{+}$ is continuous from the right, it follows there exists $\delta > 0$ so that
\begin{equation} \label{e:c2gap}
0 < \frac{m_{\alpha} - f^{\prime}_{+}(0)}{2}  \le m_{\alpha} - \partial f(t) \le m_{\alpha} f^{\prime}_{+}(0) \quad \forall ~ 0 \le t \le \delta.
\end{equation}

Define $I_{C}(x_{0})= \{ h \in H_{C} : |(x_{+}(h) -x_{0}) \cdot \theta_{C})| < \delta \}$. For $h,\tilde{h} \in H_{C}$ write $x_{+}(h), x_{+}(\tilde{h})$ as in \eqref{e:graphrep}, where without loss of generality $0 \le (x - x_{0} \cdot \theta_{C}) \le (\tilde{x} - x_{0} \cdot \theta_{C}) < \delta$. 

Then, proceeding identically to Case 1, we get inequality \eqref{e:lipboundcomp} from which we deduce
$$
\left| \theta_{\alpha} \cdot \left(\frac{x_{+}(h) - x_{+}(\tilde{h})}{h-\tilde{h}} \right) \right|^{2} \le \frac{c_{0}}{| \theta_{C}^{\perp} \cdot e_{2}|^{2}} \frac{2}{\left|f_{+}^{\prime}(\delta) - m_{\alpha} \right|^{2}},
$$
where the extra $2$ comes from the $\frac{1}{2}$ in \eqref{e:c2gap}. If necessary, we can again symmetrize $I_{C}(x_{0})$ to create an open interval around $h_{0}$ where \eqref{e:lipbound} holds. The case $f^{\prime}_{-}(0) > m_{\alpha}$ follows similarly, using left continuity of $f^{\prime}_{-}$ in place of right continuity of $f^{\prime}_{+}$.

\underline{Case 3:} $m_{\alpha} = f^{\prime}_{-}(0)$ or $f^{\prime}_{+}(0) = m_{\alpha}$ and $e_{2} \cdot \theta_{C}^{\perp} \neq 0$.

\begin{figure} 
\begin{tikzpicture}
\draw [thick] (-3,0)--(3,0) node[right, below] {$L_{x_0}$};

\node at (0,0) [circle,fill,inner sep=1.5pt]{};
\node[above,right] at (2,1) {$C = \textrm{graph}(f)$};

\draw [domain=0:2.5, variable=\x, red] [->]
plot ({\x}, { 1/2 * \x*\x	});
\draw [domain=-2.5:0, variable=\x, red] [->]
plot ({\x}, { -\x	});
\draw [thick,blue] [->] (-0.5,-0.5)--(0.5,-0.5) node[below] at (0,-0.5) {$\vec{\alpha}$};
\draw (0,0) circle node[below] {$x_0$};If
\end{tikzpicture}
\caption{ \label{f:2.1and2.3}
To the right of $x_{0}$ we see Case 3. At $x_{0}$ holds $f^{\prime}_{+}(0) = m_{\alpha}$. If $\theta_{\alpha} \not \in \cA_{\sigma}(E)$, the quantitative convexity guarantees that the slope of $f$ grows quickly.
To the left of $x_{0}$, we again see Case 1. }
\end{figure}
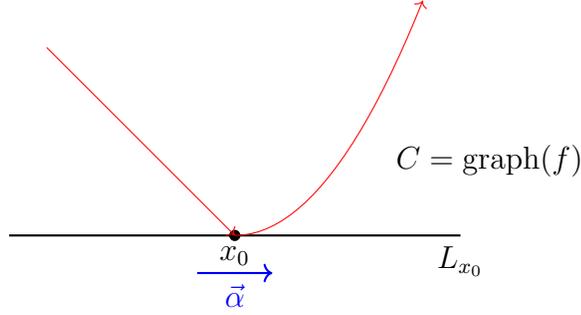

For this case, the main idea is starkly different from the others. We cannot bound $|\partial f(t) - m_{\alpha}|$ below as in Case 1 and 2. Instead, we use the quantitative convexity of $f$ to show that, as $t$ moves away from $0$, the derivative $f^{\prime}(t)$ moves away from $m_{\alpha}$ at least as fast as some fixed rate. The inverse of this rate then produces the desired upper-bounds on the rate of movement of $x(h) \cdot \alpha$, see Figure \ref{f:2.1and2.3}.

We start with the case $f^{\prime}_{+}(0) = m_{\alpha}$. Since $\alpha \not \in \cD_{\sigma}(E)$ by hypothesis, it follows that $\psi$ is $\sigma$-convex in a neighborhood of $x_{0}$.

Define $\tilde{f}(t) = f(t) - t m_{\alpha}$. Then $\tilde{f}$ is also $\sigma$-convex in a neighborhood of $0$ by convexity of $t \mapsto - t m_{\alpha}$. Since $m_{\alpha} = f^{\prime}_{+}(0)$ it follows $0 \in \partial \tilde{f}(0)$. Moreover, $x_{0} \in L_{x_{0}} \cap C$ implies $f(0) = \tilde{f}(0) = 0$.  Let $g = \tilde{f}|_{\{t \ge 0\}}^{-1}$. As $f$ is strictly increasing and continuous, also the function $g$ is continuous and defined on some half-open interval. Let $x_{+}(h)$ be as in the previous two subcases. Without loss of generality, by shrinking the domain of $\psi$ if necessary, we can suppose $\psi$ is $\sigma$-convex. Let $I_{C}(x_{0}) \subset H_{C}$ correspond to those $h$ so that $L^{h} \cap graph( \psi(x_{0} + t \theta_{C})|_{t \ge 0})$ has one point. Then $I_{C}(x_{0})$ is a half-open interval containing $h_{0}$.  
Moreover, on $I_{C}(x_{0})$\footnote{For comparison with Remark \ref{r:simp}(3), $g = G^{-1}$.}
\begin{equation} \label{e:c3alpha}
\alpha \cdot (x_{+}(h) - x_{+}(\tilde{h})) = \sqrt{1 + |m_{\alpha}|^{2}} \left[ g \left( \frac{h - h_{0}}{e_{2} \cdot \theta_{x_{0}}^{\perp}} \right) - g \left( \frac{\tilde{h} - h_{0}}{e_{2} \cdot \theta_{x_{0}}^{\perp}} \right) \right].
\end{equation}
By \eqref{e:gbound} of Proposition \ref{p:scxfunc} it follows
\begin{align} \label{e:conv_bound}
\limsup_{h \to \tilde{h}} \left| \frac{  g \left( \frac{h - h_{0}}{e_{2} \cdot \theta_{x_{0}}^{\perp}} \right) - g \left( \frac{\tilde{h} - h_{0}}{e_{2} \cdot \theta_{x_{0}}^{\perp}} \right)}{\frac{h_{1} - h_{2}}{e_{2} \cdot \theta_{x_{0}}^{\perp}} } \right| & \le c^{-\frac{1}{\sigma}} \sigma^{\frac{-1}{\sigma^{\prime}}} \left(\frac{|\tilde{h}-h_{0}|}{e_{2} \cdot \theta_{x_{0}}^{\perp}}\right)^{\frac{-1}{\sigma^{\prime}}},
\end{align}
where $\sigma^{\prime}$ is the H\"older conjugate of $\sigma$. Combining \eqref{e:c3alpha} and \eqref{e:conv_bound} yields
\begin{equation} \label{e:holderbound}
\limsup_{h \to \tilde{h}} \left| \alpha \cdot \left( \frac{x_{+}(h) - x_{+}(\tilde{h})}{h-\tilde{h}} \right) \right| \le c_{1} |\tilde{h} - h_{0}|^{\frac{-1}{\sigma^{\prime}}},
\end{equation}
for some constant $c_{1} = c_{1}(C,\alpha,\sigma)$. We can suppose $I_{C}(x_{0})$ is an open interval by symmetrization and possibly making our estimates worse. In the case $f_{-}^{\prime}(x) = m_{\alpha}$, we bound $x_{-}$ similarly, and most readily by defining $\tilde{f}(-t) = f(t) + t m_{\alpha}$  for $t > 0$.\\[12pt]
We now translate our bounds on the individual points in $L^{h_{0}} \cap \partial E$ into information about the regularity of $\tau_{E,\alpha}$. Fix $h_{0} \in [0,1]$, $C \in \cC$ so that $C \cap L^{h_{0}} \neq \emptyset$. Consider the function 
$$
F_{C}(h) =  \sum_{x(\delta) \in C \cap L^{h}} \min_{x_{0} \in C \cap L^{h_{0}}} \theta_{\alpha} \cdot (x(\delta)- x_{0})
$$
which measures the displacement in the $\theta_{\alpha}$-direction of points in $C \cap L^{h}$ from the closest points in $C \cap L^{h_{0}}$. By the preceding casework, each $x_{0} \in C \cap L^{h_{0}}$ has some interval $I_{C}(x_{0})$ on which $|\theta_{\alpha} \cdot (x_{\pm}(h) - x_{0})|$ is bounded by \eqref{e:lipbound} or \eqref{e:holderbound}.

Consider the open interval $I_{C}(h_{0}) \defeq \cap_{x_{0} \in C \cap L^{h_{0}}} I_{C}(x_{0})$. 
It follows from \eqref{e:lipbound} and \eqref{e:holderbound} that $F_{C}$ is locally Lipschitz on $I_{C}(h_{0}) \setminus \{h_{0}\}$ and in particular $F_{C}$ is differentiable almost everywhere on $I_{C}(h_{0})$. Moreover, \eqref{e:lipbound} and \eqref{e:holderbound} imply 
\begin{equation} \label{e:fprime}
|F_{C}^{\prime}(h)| \le \sum_{x_{0} \in C \cap L^{h}} A_{C,x_{0}}  +B_{C,x_{0}} |h-h_{0}|^{\frac{-1}{\sigma^{\prime}}} \quad a.e.~ h \in I_{C}(h_{0}).
\end{equation}
where $A_{C,x_{0}}$ and $B_{C,x_{0}}$ depend upon which case $C,x_{0}$ fall into. 

We now define the interval $I(h_{0}) = \cap_{C \in \cC} I_{C}(h_{0})$ which is open by finiteness of $\cC$. Moreover, $h, \tilde{h} \in I(h_{0})$ implies
$$
|\tau_{E,\alpha}(h) - \tau_{E,\alpha}(\tilde{h})| \le \sum_{\substack{C \in \cC \\ C \cap L^{h_{0}} \neq \emptyset}} |F_{C}(h) -F_{C}(\tilde{h})|.
$$

Consequently, $\tau_{E,\alpha} |_{I(h)}$ is differentiable almost everywhere, and by \eqref{e:fprime}, 
\begin{equation} \label{e:localtau}
|\tau_{E,\alpha}^{\prime}(h)| \le \sum_{\substack{C \in \cC \cap \cL\\C \cap L^{h} \neq \emptyset}} |F_{C}^{\prime}(h)| \le C_{h_{0}} (1 + |h-h_{0}|^{-\frac{1}{\sigma^{\prime}}} )
\end{equation}
holds for almost every $h \in I(h_{0})$. The dependence on $h_{0}$ on the right hand side of \eqref{e:localtau} comes from the dependence of $A_{C,x_{0}}$ and $B_{C,x_{0}}$ for each $x_{0} \in C \cap L^{h}$ and for each $C$ with $C \cap L^{h_{0}} \neq \emptyset$. Still, the dependence of \eqref{e:localtau} on $h_{0}$ does not prevent $\tau_{E,\alpha}^{\prime} \in L^{s}(\T)$. Indeed, since each $h_{0} \in [0,1]$ has some open interval on which \eqref{e:localtau} holds, and $[0,1]$ is compact, there exist finitely many $\{h_{k}\}_{k=1}^{M}$ so that $\{ I(h_{k})\}_{k=1}^{M}$ covers $[0,1]$. Since $\tau_{E,\alpha}$ is differentiable almost everywhere on each $I(h_{k})$ it is differentiable almost everywhere on $\T$. Consequently, writing $\overline{C} \defeq \max_{k} C_{h_{k}}$ it follows
\begin{align*}
\int_{0}^{1} | \tau_{E,\alpha}^{\prime}(h)|^{s} \mathrm{d}h & \le \sum_{k=1}^{M} \int_{I(h_{k})}C_{h_{k}}^{s} (1 + |h-h_{0}|^{-\frac{1}{\sigma^{\prime}}})^{s} \mathrm{d} h \le 2^{s} \overline{C}^{s} M\int_{-1}^{1} (1 + |h|^{\frac{-s}{\sigma^{\prime}}}) \mathrm{d} h,
\end{align*}
which is finite if and only if $s < \sigma^{\prime}$. Since $|\tau_{E,\alpha}| \le 1$, this completes the proof that $\tau_{E,\alpha} \in W^{1,s}(\T)$ for all $s < \sigma^{\prime}$.
\end{proof}

\bibliographystyle{alpha}
\bibliography{Sources}
\end{document}